\documentclass[preprint,12pt,number]{elsarticle}

\usepackage{amssymb}
\usepackage{amsthm}
\usepackage{amsmath}
\usepackage{subfig}
\usepackage{hyperref}
\usepackage{cleveref}

\newtheorem{definition}{Definition}
\newtheorem{theorem}{Theorem}
\newtheorem{lemma}{Lemma}
\newtheorem{corollary}{Corollary}

\newtheorem{proposition}{Proposition}

\newcommand{\R}{\mathbb{R}}
\newcommand{\tr}{\tilde{r}}
\newcommand{\Z}{\mathcal{Z}}
\journal{   }

\begin{document}

\begin{frontmatter}


 \tnotetext[label1]{The authors gratefully acknowledge the ARC Discovery grant DP180101602 and many valuable discussions with members of the research team Y.Nazarathy, T.Taimre, H.Jansen and S.Streipert of that project.}
 
  \tnotetext[label2]{Key Words: Threshold Risk, Tail Probabilities, Polynomial Perturbations, Roots of Polynomials, Discriminant and Puiseux Series. }
 

\title{Hidden  Equations of Threshold Risk}


\author[label1]{Vladimir V. Ejov}
\author[label2]{Jerzy A. Filar}

\address[label1]{College of Science and Engineering, Flinders University, South Australia, Australia}
\address[label2]{Centre for Applications in Natural Resource Mathematics, School of 
Mathematics and Physics,The University Of Queensland, Queensland, Australia}

\author[label2]{Zhihao Qiao}

\begin{abstract}
We consider the problem of sensitivity of threshold risk, defined as the probability of a function of a random variable falling below a specified threshold level $\delta >0.$  We demonstrate that for polynomial and rational functions of that random variable there exist at most finitely many risk critical points.  The latter are those special values of the threshold parameter for which rate of change of risk is unbounded as  $\delta$ approaches these threshold values. We characterize candidates for risk critical points as zeroes of either the resultant of a relevant  $\delta-$perturbed polynomial, or of its leading coefficient, or both. Thus the equations that need to be solved are themselves polynomial equations in  $\delta$ that exploit the algebraic properties of the underlying polynomial or rational functions.  We name these important equations as "hidden equations of threshold risk".   

\end{abstract}

% \begin{highlights}
% \end{highlights}

% \begin{keyword}
% %% keywords here, in the form: keyword \sep keyword

% %% PACS codes here, in the form: \PACS code \sep code

% %% MSC codes here, in the form: \MSC code \sep code
% %% or \MSC[2008] code \sep code (2000 is the default)

% \end{keyword}

\end{frontmatter}


\section{Introduction and Motivation}

This paper is motivated by the dual notions of ``tipping points" and ``risk sensitivity" frequently arising in society's interactions with the natural environment.  On some level the problem is at least as old as the history of agriculture with farmers being concerned about rainfall falling below some acceptable level,  or onset of frost; a prototypical tipping point for successful cultivation of certain crops. 

More recently, concerns about adverse climate change induced global warming have focused on the level of such warming exceeding thresholds such as 1.5 or 2.0 degrees C, by the year 2030 or 2050.  Similarly, in the area of sustainable management of fisheries, regulators often consider a fishery secure if the biomass of the harvested species does not fall below a certain percentage (e.g., $60\%$) of the virgin biomass.   
From the perspective of mathematical modelling of these concerns we first recognise two essential features:
\begin{itemize}
    \item often the variable that we are most interested in (e.g. harvest yield, or fish stock) depends essentially on at least one random variable;
    \item the tipping point is, perhaps, most naturally represented as a ``special value" of some parameter. In particular, a value such that if the variable of interest falls below (or above) that value, this is considered to be a ``high risk" situation.   
\end{itemize}
The use of quotation marks in that last point suggests that there is a need to make these phrases precise so as to be able to analyse them rigorously. In particular, there are already several alternative mathematical formulations of risk often stemming from actuarial science, finance or engineering. However, in this paper, we take the position that the simple threshold risk is both appropriate and already challenging in the context of the management of natural resources such as fisheries. Conceptually, this risk is modelled as a tail probability$$P(h(\text{random variable}) < \delta), $$
where $\delta$ is the threshold parameter, and $h(\cdot)$ is a given function.

At first sight, this formulation of risk may appear to correspond to a problem fully solved by mathematical statisticians and probabilists. In particular, it is a problem extensively studied in the context of extreme value theory (e.g., see  \cite{embrechts_modelling_1997}), financial mathematics (e.g., see \cite{GOURIEROUX2000225}, \cite{kluppelberg_ruin_1998}) and large deviation theory (e.g., see \cite{varadhan_large_1984}). These approaches focus primarily on asymptotic properties of tail probabilities of certain classes of distributions. However, our approach is essentially different in the sense that we explore the parametric sensitivity of the threshold risk induced by the algebraic form of the function of the random variable that is of interest.   

Indeed, recent applied studies such as \cite{ben-tovim_hospitals_2018} and \cite{filar2020}, indicate that the threshold risk may exhibit high sensitivity to the choice of model parameters, including the threshold parameter. The latter arose in two quite disparate  contexts of hospital management and fishery population models. This leads us to a more formal definition and analysis of threshold risk and critical values of the threshold which are natural candidates for tipping points. This is taken up in the next section.    

\section{Risk sensitivity of threshold probability with polynomials of random variables }
The most general one dimensional problem we will consider here is one where the random variable that is of main interest to us is actually a known rational function $h(X)$ of another random variable $X$ whose cumulative distribution function(cdf) $F(x)$ is also assumed to be known. In this paper we assume that $X$ is an absolutely continuous random variable, hence the density function $f(x)$ is well-defined (see also Remark 2 in Section 3). 
We begin the analysis with a simpler case where $h(X)=p(X),$ a known polynomial in $X.$
\begin{definition}[Risk with one polynomial function]
Let $X$ be a random variable and $\delta \in \R$ and consider a polynomial $p(X) = p_0+p_1X+p_2X^2+\cdots+p_n X^n.$ The {\em threshold risk} probability is
\begin{equation}\label{eq:2.1}
    R(\delta) = P \Big( p(X)< \delta \Big),
\end{equation}
 where $\delta$ is a real valued parameter denoting the {\em threshold}.
 \end{definition}
 Of course, in some applications the inequality in \eqref{eq:2.1} would be reversed. More generally, the {\em threshold} could be a multiple of another polynomial function $q(X)$. In such a case, the threshold risk definition is extended as follows.   
 
 \begin{definition}[Risk with two polynomial functions]
 With the same quantities as in Definition 1 and $q(X) =q_0+q_1X+q_2X^2+\cdots+q_m X^m$
 \begin{equation}\label{eq:2.11}
    R(\delta) = P \Big( p(X)< \delta q(X) \Big) =P \Big( p(X)- \delta q(X) <0 \Big).
\end{equation}
\end{definition}

Note that in both \eqref{eq:2.1} and \eqref{eq:2.11}, we could have defined a $\delta$-perturbed polynomial $p_\delta(X) = p(X) - \delta q(X)$ and expressed the threshold risk as $$R(\delta) = P \Big( p_{\delta}(X)< 0 \Big).$$ Naturally, in the case of \eqref{eq:2.1}, $q(X)$ is identically equal to 1. 

 A change in risk as $\delta_0$ changes to $\delta_1$ will be measured by the ratio
\begin{equation}\label{eq:2.2}
    S(\delta_0,\delta) = \frac{|R(\delta_0)-R(\delta)|}{|\delta_0-\delta|} \approx  |R'(\delta_0)|,
\end{equation}
if the derivative exists and $\delta$ is close to $\delta_0.$ 

\begin{definition}
Threshold risk sensitivity is now defined as follows.
\begin{enumerate}
\item The risk sensitivity at $\delta_0$ is defined as the absolute value of the derivative  $R'(\delta_0)$, if it exists. 
\item If $|R'(\delta_0)|$ is infinite or undefined, then $\delta_0$ is a candidate risk critical point.
\item We say $\delta_0$ is a risk critical point if there does not exist a neighbourhood $\mathcal{N}$ of $\delta_0$ such that $S(\delta_0,\delta)$ is uniformly bounded for all $\delta \in \mathcal{N}.$ 
\end{enumerate}
\end{definition}

This is related to (but not the same) as the {\em hazard function} used in demography and actuarial science. The latter considers the ratio of the probability density function at $\delta_0$ to the probability of exceeding that threshold. 

To analyze the polynomial threshold risk in more detail it will be necessary to consider the real roots of the underlying polynomial. Let $r_1(\delta)\leq r_2(\delta) \ldots \leq r_{n_{1}}(\delta)$ be the real roots of $p_\delta(X)=0$ for $n_1 \leq n$. We can partition $\R$ into union of following intervals
\begin{align*}
I_0(\delta) &= \big(-\infty,r_1(\delta)\big),\\
I_1(\delta) &= \big[r_1(\delta),r_2(\delta)\big),\\
I_2(\delta) &= \big[r_2(\delta),r_3(\delta) \big),\\
\vdots\\
I_{n_1}(\delta) &= \big[r_{n_1}(\delta),\infty \big),
\end{align*}
where we observe that the sign of the polynomial $p_\delta(X)$ cannot change inside any of the intervals $I_j$.
Let $\mathcal{J}^{-}(\delta) = \{ j| p_{\delta}(x)\leq 0, \; \text{if} \; \;x \in I_j\}$. The threshold risk can now be expressed as 
\begin{equation}\label{eq:2.3}
    R(\delta)= \sum_{j \in \mathcal{J^-}(\delta)} R(I_j(\delta)) = \sum_{j \in \mathcal{J^-}(\delta)} \int_{x\in I_j(\delta)} f(x)dx  = \sum_{j\in \mathcal{J^-(\delta)}} [F(r_{j+1}(\delta)) - F(r_j(\delta))],
\end{equation}
where $F(x) = \int_{-\infty}^{x} f(u) du.$ 
Hence its derivative with respect to  $\delta$ is given by
\begin{equation}\label{eq:2.4}
    R'(\delta) = \sum_{j \in \mathcal{J^-}(\delta)} [f(r_{j+1}(\delta))r_{j+1}'(\delta) - f(r_j(\delta))r_{j}'(\delta)],
\end{equation}
whenever the derivatives of these roots at $\delta$ exist.
% \begin{lemma}
% If we perturb the polynomial $p_\delta(X)$ by adding or subtracting a sufficiently small $\epsilon >0$, the repeated roots of $p_\delta(X)$ with even multiplicity will be replaced by two distinct roots. However, the repeated roots of $p_\delta(X)$ with odd multiplicity will be replaced by a single root with odd multiplicity. 
% \end{lemma}
% \begin{proof}
% Obvious from the nature of the shift of the graph of the original polynomial when perturbed by $\epsilon$.
% \end{proof}
Let $\text{Dis}(p_{\delta}(X))$ denote the discriminant of the polynomial $p_{\delta}(X).$\\

$\textbf{Remark 1:}$
It will be seen that the discriminant is a polynomial in $\delta$. This is what we refer to as "the hidden polynomial". Consider the equation \begin{equation}\tag{$*$}
\label{eq:dis}
    \text{Dis}(p_\delta(X)) =0.
\end{equation}
In the following proposition, we show that the hidden equation \eqref{eq:dis} is that of finding the roots of a polynomial in $\delta$ with finite order. 

\begin{proposition}
For the case where $p_\delta(X) = p(X)-\delta$, the hidden polynomial $\text{Dis}(p_\delta(X))$ is a polynomial in $\delta$ of order not greater than $n-1$. 
\end{proposition}
\begin{proof}
Using Lemma 6 in the Appendix, we have 
\begin{equation}\label{eq:2.5}
\begin{split}
      &\text{Dis}(p_\delta(X)) =\frac{(-1)^{n(n-1)/2}}{p_n}\text{Res}(p_{\delta}(X),p'_{\delta}(X))\\ =&\frac{(-1)^{\frac{n(n-1)}{2}}}{p_n}\text{Det}\begin{bmatrix}p_n&p_{n-1}&\cdots&\cdots&-\delta&0&\cdots&0\\
    0&p_{n}&\cdots&\cdots&p_1&-\delta&\cdots&0\\
        \vdots&\vdots&\ddots&\ddots&\vdots&\vdots&\ddots&-\delta\\
        np_n&(n-1)p_{n-1}&\cdots&\cdots&p_1&0&\cdots&0\\
    0&np_{n}&\cdots&\cdots&\cdots&p_1&\cdots&0\\
        \vdots&\vdots&\ddots&\ddots&\vdots&\vdots&\ddots&p_1
    \end{bmatrix}.
\end{split}
\end{equation}
Note that there are only $n-1$ columns containing $-\delta$. Therefore, $\text{Dis}(p_\delta(X))$ is a polynomial in $\delta$ of order not greater than $n-1$.
\end{proof}

The next theorem shows that zeroes of the discriminant play an important role in our analysis. In particular, the theorem will show these zeroes contain the candidates of risk critical points.   

\begin{theorem}
Let $\mathcal{Z}(p_{\delta}(X)) = \{ \delta \; | \text{Dis}(p_{\delta}(X)) = 0 \}$. It follows that
\begin{enumerate}
    
    \item[(i)] If the set $C$ of critical risk points is non-empty, then it is a subset of $\mathcal{Z}(p_{\delta}(X))$, 
        \item[(ii)]   There exist at most $n-1$ critical risk points, where $n$ is the degree of $p_{\delta}(X)$.
\end{enumerate}
\end{theorem}

\begin{proof}   
 Proof of (i). We apply Theorem 3 in the Appendix. If $\delta_0 \in \mathcal{Z}(P_{\delta}(X))$, then for some $r_j(\delta_0)$, it has a root expansion with branching order $n' \leq n $, the root $r_{j}(\delta)$ has a Puiseux series representation 
$$ r_j(\delta) = \sum_{k=0}^\infty c_{jk}(\delta - \delta_0)^{k/n'},$$ and 
$$ r_j'(\delta) = \sum_{k=1}^\infty c_{jk}\frac{k}{n'}(\delta - \delta_0)^{k/n'-1}.$$ 
If $n'>1$, $k/n' -1 <0$ if $k<n'$ , therefore $\lim_{\delta \rightarrow \delta_0}(\delta - \delta_0)^{k/n' -1} \rightarrow \infty$. However, not all $\delta_0 \in \mathcal{Z}(P_\delta(X))$ are risk critical points. For instance, $r_j(\delta_0)$ could be a repeated root but not in the support of the distribution of $X$.

Similarly, if $\delta_0 \notin \mathcal{Z}(P_{\delta}(x))$, then the root $r_j(\delta_0)$ has an expansion with a power series 
$$r_j(\delta) = \sum_{k=0}^{\infty} c_{jk}(\delta -\delta_0)^k$$ and
$$r_j'(\delta) = \sum_{k=1}^{\infty} kc_{jk}(\delta -\delta_0)^{k-1}$$
and we have $\lim_{\delta \rightarrow \delta_0} r'_{j}(\delta) < \infty.$ Hence $\delta_0$ is not a risk critical point. Therefore, $C$ is a subset of $\mathcal{Z}(p_{\delta}(x))$.

Proof of (ii) follows immediately from Proposition 1.  
% For example, for the cubic polynomial $p_\delta(X) = a_3X^3+a_2X^2+a_1X+\delta$, we have 
% \begin{equation*}
%     Dis(p_\delta(X)) =\frac{(-1)^{3(3-1)/2}}{a_n}Res(p,p') =  \frac{-1}{a_3}\text{Det}\begin{bmatrix}a_3&a_{2}&a_{1}&\delta&0\\
%     0&a_{3}&a_{2}&a_1&\delta\\
%     3a_3&2a_2&a_{1}&0&0\\
%     0&3a_3&2a_2&a_{1}&0\\
%     0&0&3a_3&2a_2&a_{1}\\
%     \end{bmatrix}.
% \end{equation*}
\end{proof}

\textbf{Example 1:} Let $p_{\delta_0}(X) = X^2-\delta_0$ and $X\sim \mathcal{N}(0,1^2)$.
Here $\text{Dis}(p_{\delta_0}(X)) = 4\delta_0$, hence  $\mathcal{Z}(p_{\delta_0}(X))= \{ 0\}.$  If $\delta_0 =0$, $r_1(\delta_0)=0$ is a root with even multiplicity. By Theorem 1, it is a candidate risk critical point. If we perturb $\delta_0$ to $\delta = \delta_0+\varepsilon$, $p_\delta(X)$ has roots $r_1(\delta) = -\sqrt{\delta}, r_2(\delta) = \sqrt{\delta}$ and $r_1'(\delta) = - \frac{1}{2\sqrt{\delta}},r_2'(\delta) =  \frac{1}{2\sqrt{\delta}}$. 
 Recalling the density function of standard normal distribution we see that equation \eqref{eq:2.4} implies that the rate of change of the threshold risk is now given by $R'(\delta) =  \frac{1}{2\sqrt{\delta}}\frac{1}{\sqrt{2\pi}}e^{-0.5\delta} - (- \frac{1}{2\sqrt{\delta}}\frac{1}{\sqrt{2\pi}}e^{-0.5\delta}) = \frac{1}{\sqrt{\delta}}\frac{1}{\sqrt{2\pi}}e^{-0.5\delta} $. Clearly, the latter diverges as  $\delta \to 0$. It is now easy to see that $\delta_0 =0$ is a risk critical point.

\section{Repeating and non-repeating root decomposition of constant perturbation}
In the previous section, we derived a theorem which identifies candidates for risk critical points. However, since the sensitivity of the risk also depends on the coefficients of the root expansion as well as the value of the density function, we need to further decompose the intervals in \eqref{eq:2.3}. In particular, we shall separate the contribution to the threshold risk from repeating and non-repeating roots. 
\begin{definition}
Let $\mathcal{J}^-(\delta) = J^{-}_{r}(\delta) \bigcup J^{-}_{n}(\delta)$, where $J^{-}_{r}(\delta)$ and $J^{-}_{n}(\delta)$ are defined by
\begin{enumerate}
    \item $j \in J^{-}_r(\delta)$ if and only if $j \in \mathcal{J}^-(\delta)$ and $I_j(\delta) = (r_j(\delta), r_{j+1}(\delta)) \in J^-(\delta)$, and at least one of $r_j(\delta)$ and $r_{j+1}(\delta)$ is a repeated root of $p_\delta(X).$ 
    \item $j \in J^{-}_n(\delta)$ if and only if $j \in \mathcal{J}^-(\delta)$ and $I_j(\delta) = (r_j(\delta), r_{j+1}(\delta)) \in J^-(\delta)$, and both $r_j(\delta)$ and $r_{j+1}(\delta)$ are non-repeating roots of $p_\delta(X).$ 
\end{enumerate}
\end{definition}

Using the above definition, \eqref{eq:2.3} can be partitioned as follows
\begin{equation}\label{eq:2.1.1}
    R(\delta) = R_r(\delta)+R_n(\delta) =  \sum_{j\in J^{-}_r(\delta)} R(I_j(\delta)) +  \sum_{j\in J^{-}_n(\delta)} R(I_j(\delta)).
\end{equation}
By Theorem 1, critical points must be among the zeroes of the discriminant of $p_\delta(X)$. However, the discriminant is zero only if the resultant is zero since the leading coefficient $p_n\neq 0$. Now the resultant is zero only if $p_\delta(X)$ has repeated roots. If $p_\delta(X)$ has no repeated roots, then the hidden equation \eqref{eq:dis} cannot be satisfied. Therefore $J_r^{-}(\delta)$ is empty and $\delta$ is not a risk critical point.

\begin{lemma}
For the non-repeated root component $R_n(\delta)$, we have 
$$\lim_{\delta \rightarrow \delta_0} \frac{R_n(\delta)-R_n(\delta_0)}{|\delta-\delta_0|} = c,$$
for some finite scalar $c$.
\end{lemma}

\begin{proof}
We have from \eqref{eq:2.1.1}

\begin{equation}\label{eq:2.1.2}
    R_n(\delta_0) = \sum_{j\in J^{-}_n(\delta_0)} R(I_j(\delta_0)) =\sum_{j\in J^{-}_n(\delta_0)} F(r_{j+1}(\delta_0)) - F(r_{j}(\delta_0)).
\end{equation}
Since non-repeating roots have multiplicity 1, they have a root expansion in a small neighbourhood of $\delta_0$ as $r_j(\delta) = \sum_{k=0}^{\infty}c_{jk}(\delta-\delta_0)^k.$ Understanding that $c_{j0} = r_j(\delta_0)$, implies $r_j(\delta) - r_j(\delta_0) = \sum_{k=1}^{\infty}c_{jk}(\delta-\delta_0)^k$. If we only consider one interval $I_j(\delta)=(r_j(\delta),r_{j+1}(\delta)) \in J_n^{-}(\delta)$, then 
\begin{align*}
R(I_j(\delta)) - R(I_j(\delta_0))  &= \big[F(r_{j+1}(\delta))-F(r_{j+1}(\delta_0))\big] -\big[F(r_{j}(\delta))- F(r_{j}(\delta_0))\big]\\
& \approx f(r_{j+1}(\zeta))(r_{j+1}(\delta) - r_{j+1}(\delta_0)) - f(r_{j}(\zeta))(r_{j}(\delta) - r_{j}(\delta_0)) \\
& =  f(r_{j+1}(\zeta))\sum_{k=1}^{\infty}c_{(j+1)k}(\delta-\delta_0)^k - f(r_{j}(\zeta))\sum_{k=1}^{\infty}c_{jk}(\delta-\delta_0)^k, 
\end{align*}
for some $\zeta$ between $\delta$ and $\delta_0$.
Hence, we have 
$$\lim_{\delta \rightarrow \delta_0} \frac{R(I_j(\delta))-R(I_j(\delta_0))}{|\delta-\delta_0|} =  f(r_{j+1}(\zeta))c_{(j+1)1}-f(r_{j}(\zeta))c_{j1},$$
which is bounded.
Since this holds for every interval in $J_n^{-}(\delta)$, the result follows from \eqref{eq:2.1.2}.
\end{proof}
The above lemma shows that in the case of a threshold $\delta_0$ such that the roots of polynomial $p_{\delta_0}(X)$ are all non-repeating, $\delta_0$ cannot be a risk critical point. However the case where the polynomial $p_{\delta_0}(X)$ has repeated roots is more complicated. Below, we first demonstrate the distinction between the case with the multiplicity of the repeated root being odd and even. Then we analyse the case where the interval $I_j(\delta)$ in $J_r^{-}(\delta)$ contains a repeated and a non-repeated root.

\begin{lemma}
Consider the perturbation of the form $p_\delta(X) = p(X) - \delta$, for $\delta - \delta_0>0$ and sufficiently close to $\delta_0$, where  $r_j(\delta_0)$ is a repeated root of $p_{\delta_0}(X)=0$.
Then the order of the branching point of a repeated root is exactly 2 if the multiplicity of the root is even, and exactly 1 if the multiplicity of the root is odd. 
\end{lemma}
\begin{proof}
The graph of the polynomial $p_{\delta}(X)$ is simply the graph of $p_{\delta_0}(X)$ shifted down by $\delta-\delta_0$. If the multiplicity of the root $r_j(\delta_0)$ is odd, the polynomial $p_\delta(X)$ crosses the x-axis at the corresponding root $r_j(\delta)$. Hence there is only one branch of that root. If on the other hand, the multiplicity of the root is even, the polynomial $p_{\delta_0}(X)$ touches x-axis at $r_j(\delta_0)$, but $p_\delta(X)$ will have two distinct roots: one to the left of $r_j(\delta_0)$ and one to its right. Hence there are two branches of the root. 
\end{proof}

\begin{proposition}
Consider the perturbation of the form $p_\delta(X) = p(X) - \delta$, for $\delta>\delta_0$ and sufficiently close to $\delta_0$. Suppose $r_j(\delta_0)$ is the only repeated root of $p_{\delta_0}(X)$ and $f(r_j(\delta_0))>0$. We have the following cases,
\begin{enumerate}
    \item[(i)] If $r_j(\delta_0)$ has odd multiplicity, then it is not a risk critical point, 
    \item[(ii)] If $r_j(\delta_0)$ has even multiplicity and $p_{\delta_0}(X) > 0$ in a deleted neighbourhood of $r_j(\delta_0)$, then $\delta_0$ is risk critical point. 
    \item[(iii)] If $r_j(\delta_0)$ has even multiplicity and $p_{\delta_0}(X) < 0$ in a deleted neighbourhood of $r_j(\delta_0)$, then $\delta_0$ is not a risk critical point.
\end{enumerate}
\end{proposition}
% \begin{lemma}
% For $J_r^{-}(\delta)$ component, we have 
% $$\lim_{\delta \rightarrow \delta_0} \frac{R_r(\delta)-R_r(\delta_0)}{|\delta-\delta_0|}$$
% diverges with order $|\delta-\delta_0|^{1/m_{h}^{'} -1}$, 
% where $m_{h}^{'}$ is the highest order of the branching point among the repeated roots. 
% \end{lemma}

\begin{proof}
(i) If $r_j(\delta_0)$ has odd multiplicity, by Lemma 2, the order of the branching point of $r_j(\delta_0)$ is $n'=1$. Therefore, by Theorem 3(b) of the Appendix, it has the root expansion in a small neighbourhood of  $\delta_0$ as $r_j(\delta) = \sum_{k=0}^{\infty}c_{jk}(\delta-\delta_0)^{k}.$ Understanding that $c_{j0} = r_j(\delta_0)$, it follows that $r_j(\delta) - r_j(\delta_0) = \sum_{k=1}^{\infty}c_{jk}(\delta-\delta_0)^{k}$, and without loss of generality, we have that $r_{j+1}(\delta)$ is a non-repeated root of $p_\delta (X) =0$. 
It now follows that
$$\lim_{\delta \rightarrow \delta_0} \frac{R(I_j(\delta))-R(I_j(\delta_0))}{|\delta-\delta_0|} =  f(r_{j+1}(\zeta))c_{(j+1)1}-f(r_{j}(\zeta))c_{j1},$$
for some $\zeta$ between $\delta$ and $\delta_0$.
The right hand side of the above is bounded.  Similarly for contributions to $R(\delta)$ for all other intervals $I_k(\delta)$ where $k\neq j$. Hence  $\delta_0$ is not a risk critical point.

(ii) Similarly, by Lemma 2, if $r_j(\delta_0)$ has even multiplicity and $p_{\delta_0}(X) > 0$ in a deleted neighbourhood of $r_j(\delta_0)$, the order of branching point is $n'=2$. Therefore, by Theorem 3(b) of the Appendix, there are two root expansions in a small neighbourhood of $\delta_0$ namely  $r_j(\delta) = \sum_{k=0}^{\infty}c_{jk}(\delta-\delta_0)^{k/2}$ and $r_{j+1}(\delta) = \sum_{k=0}^{\infty}c_{(j+1)k}(\delta-\delta_0)^{k/2}$.
It follows that $c_{j0} = c_{(j+1)0} = r_j(\delta_0)$. Next
\begin{align*}
R(I_j(\delta)) - R(I_j(\delta_0))  &= \big[F(r_{j+1}(\delta))-F(r_{j+1}(\delta_0))\big] -\big[F(r_{j}(\delta))- F(r_{j}(\delta_0))\big]\\
& \approx f(r_{j+1}(\zeta))(r_{j+1}(\delta) - r_{j+1}(\delta_0)) - f(r_{j}(\zeta))(r_{j}(\delta) - r_{j}(\delta_0)) \\
& =  f(r_{j+1}(\zeta))\sum_{k=1}^{\infty}c_{(j+1)k}(\delta-\delta_0)^{k/2} - f(r_{j}(\zeta))\sum_{k=1}^{\infty}c_{jk}(\delta-\delta_0)^{k/2},
\end{align*}
for some $\zeta$ between $\delta$ and $\delta_0$.
If we consider $\varepsilon=\delta-\delta_0 >0$, we have 
\begin{equation}\label{eq:2.1.3}
    \begin{split}
        &\lim_{\delta \rightarrow \delta_0} \frac{R(I_j(\delta))-R(I_j(\delta_0))}{|\delta-\delta_0|} \\
&\approx  \lim_{\varepsilon \rightarrow 0} \{ f(r_{j+1}(\zeta))\big[c_{(j+1)1}\varepsilon^{-1/2}+c_{(j+1)2}+c_{(j+1)(3)}\varepsilon^{3/2 -1}+\ldots\big] \\
& - f(r_{j}(\zeta))\big[c_{j1}\varepsilon^{-1/2}+c_{j2}+c_{j3}\varepsilon^{3/2 -1} +\ldots\big] \},\\
& \approx \lim_{\varepsilon \rightarrow 0}\{f(r_{j+1}(\zeta))\big[c_{(j+1)1}\varepsilon^{-1/2}+c_{(j+1)2}\big]
- f(r_{j}(\zeta))\big[c_{j1}\varepsilon^{-1/2}+c_{j2}\big]  \} .
    \end{split}
\end{equation}
Since $r_j(\delta_0)$ has even multiplicity, for $\delta>\delta_0$ and sufficiently close to $\delta_0$, one of the roots, say $r_{j+1}(\delta)$ is bigger than $r_j(\delta_0)$, and the other one is smaller. Hence if $r_{j+1}(\delta) > r_j(\delta)$, then $c_{(j+1)1}>0$ and $c_{j1}<0$. Therefore $f(r_{j+1}(\zeta))c_{(j+1)1} -f(r_{j}(\zeta))c_{j1}>0 $, and the above limit diverges as $\varepsilon \rightarrow 0$. 
Note that contributions to $R(\delta)$ for all other intervals $I_k(\delta)$ where $k \neq j$ are constant as in (i). Hence $\delta_0$ is a risk critical point.\\

(iii) If $r_j(\delta_0)$ has even multiplicity and $p_{\delta_0}(X) <0 $ in a deleted neighbourhood of $r_j(\delta_0)$, then $r_j(\delta_0)$ is a local maximum of $p_{\delta_0}(X)$. Hence, the interval $I_j(\delta_0)$ is of measure $0$ and contributes nothing to $R(I_j(\delta_0))$. When the graph of the polynomial is shifted down by $\varepsilon= \delta - \delta_0,$ there is no longer a root in a sufficiently small neighbourhood of $r_j(\delta_0)$. Hence, again, there is no contribution to $R(\delta)$ from $R(I_j(\delta))$. Thus $\delta_0$ is not a risk critical point.
\end{proof}

\begin{corollary}
Assume conditions of Proposition 2 (ii) apply to  two or more distinct repeated roots $r_1(\delta_0),\ldots,r_l(\delta_0)$,  with even multiplicity. If the density function $f(r_j(\delta_0))>0$ for at least one of these roots, then $\delta_0$ is a risk critical point. 
\end{corollary}
\begin{proof}
This is a generalization of Proposition 2. Since these even multiplicity roots are local minima, they contribute to the threshold risk sensitivity, in the limit as $\delta \to \delta_0$, as in the proof of part (ii) of Proposition 2. Furthermore, there is no such contribution from any other roots.   

% For constant perturbations, the graph of the polynomial is either moving upwards or downwards. If $r_j(\delta_0)$ is an even-multiplicity root, then it bounces off the X-axis. 
\end{proof}

$\textbf{Remark 2:}$ Note that the assumption that $X$ was an absolutely continuous random variable could be easily relaxed, but at the cost of more complicated notation and some additional technicalities.  For instance,  if $\tilde{x}$ is a discontinuity of the cdf $F(x)$ and $\delta_0$ is a threshold such that the $j^{th}$ root $r_j(\delta_0)= \tilde{x},$ then $\delta_0$ is a candidate for a risk critical point.

\section{Rational Function}
Next we analyze the situation when the underlying function of the random variable $X$ is a rational function, namely, a ratio of two polynomials. Indeed, this was the case in the motivating study \cite{filar2020}.
\begin{definition}[Risk with rational function] Let $X$ be a random variable and $\delta \in \R$. Let $p(X) =p_0+p_1X+\ldots +p_nX^n$ and $q(X) = q_0 + q_1X+\ldots q_mX^m$ be two co-prime polynomials. Let $h(X) = \frac{p(X)}{q(X)}$ and consider the threshold risk 
\begin{equation}\label{eq:3.1}
\begin{split}
R(\delta) &= P\bigg(h(X)<\delta\bigg)\\
&=P\big(p(X)-\delta  q(X)<0|q(X)>0\big)P\big(q(X)>0\big)\\
&+P\big(p(X)-\delta q(X)>0|q(X)<0\big)P\big(q(X)<0\big).
\end{split}
\end{equation}
\end{definition}
The risk sensitivity is defined as before in Definition 2.
% For other cases, we can either use Taylor's expansion to approximate the function as a polynomial. More efficiently, we can also use Pad\'e's approximation of the function as a rational polynomial. 
% \begin{definition}[Pad\'e's approximation]
% Let $f(x)$ be a function which is neither a polynomial nor a rational polynomial, if we compute its Taylor expansion at some point $x=a$ to order $M+N$, we have 
% $$f(x) \approx T_{M+N}(x) = t_0 + t_1(x-a) +t_2(x-a)^2 +\ldots + t_{M+N} (x-a)^{M+N},$$
% each of the coefficient $t_i = \frac{f^{(i)}(a)}{i!}$.
% The Pad\'e's approximation can be computed as 
% \begin{equation}\label{eq:3.3}
%  f(x) \approx \frac{\sum_{i=0}^N a_i(x-a)^i}{1+\sum_{i=1}^M b_i(x-b)^i} = T_{M+N}(x),   
% \end{equation}
% multiplying the denominators, we have the following set of coefficient relation
% \begin{align*}
% &a_0 = t_0\\
% &a_1 = t_1+t_0b_1\\
% &a_2 = t_2+t_1b_1+t_0b_2\\
% &a_3 = t_3+t_2b_1+t_1b_2+t_0b_3\\
% &\ldots \; \; \ldots
% \end{align*}
% \end{definition}
% In sequel it will become clear that a search for critical points involves finding roots of certain polynomials. Since polynomials on the numerator and denominator of the rational function approximation are of lower order than that the polynomial $T_{M+N}(x)$, the complexity of finding roots should be lower.
In this rational function case the roots of the denominator $q(X)$ will impact \eqref{eq:3.1}. To compute the threshold risk, let $\tilde{r}_1\leq \tr_2 \leq \ldots \leq \tr_{m'}$ be the real roots of polynomial $q(X)$ where $m'\leq m$. We can factor $q(X)$ as 
$$q(X) = \bigg[\prod_{d=1}^{m'}\big(X-\tr_d\big)\bigg]\tilde{q}(X).$$
Now we can partition $\R$ by these roots as 

\begin{equation}\label{eq:3.2}
    \R = \bigcup_{j=1}^{m'+1} \tilde{I}_j = (-\infty,\tr_1) \cup [\tr_1,\tr_2)\cdots \cup[\tr_{m'},\infty).
\end{equation}

Some of the intervals can have zero length if there are repeated roots for $q(X)$. Next, we define the event of interest as
\begin{equation}\label{eq:3.4}
    E = \bigg\{ x \; \bigg| \frac{p(x)}{q(x)} \leq \delta \bigg\} = \bigcup_{j=1}^{m'+1} E_j,
\end{equation}
where $E_j = E \cap \tilde{I}_j$. Then we partition the index set $J =J^+ \cup J^- =  \{1,2,\ldots,m'\}$ as
\begin{align*}
    J^+ &= \big\{ j\; | q(x)>0 \;\;\;\forall x \in \tilde{I_j}\big\},\\
        J^- &= \big\{ j\; | q(x)<0 \;\;\;\forall x \in \tilde{I_j}\big\}.
\end{align*}

The threshold risk probability can be computed as 
\begin{equation}\label{eq:3.5}
R(\delta) = P(E) = \sum_{j\in J^+}P(E_j) + \sum_{j\in J^-}P(E_j).
\end{equation}
For the polynomial function $h_\delta(X) = p(X) - \delta q(X)$, we can express the events $E_j$ in \eqref{eq:3.5} more explicitly conditioned on the sign of $h_\delta(X)$
\begin{equation}\label{eq:3.6}
\begin{split}
    &\text{if}\; \; j\in J^+,  \; E_j = \big \{ x \in \tilde{I}_j\big\} \cap \big\{ x |h_\delta(x) \leq 0 \big\},\\
        &\text{if}\; \; j\in J^-,  \; E_j = \big \{ x \in \tilde{I}_j\big\} \cap \big\{ x |h_\delta(x) > 0 \big\}.
\end{split}
\end{equation}

Let $r_1(\delta) \leq r_2(\delta) \leq \cdots \leq r_n'(\delta)$, where $n'\leq n$, be the real roots of this polynomial $h_\delta(X)$. We can factor this polynomial as $$h_\delta(X) = \bigg[\prod_{k=1}^{n'}\big(X-r_k(\delta) \big)\bigg]\tilde{h}(X),$$
and similarly, we can partition $\R$ by these roots as 

\begin{equation} \label{eq:3.61}
    \R = \bigcup_{k=1}^{n'+1} I_k(\delta) = (-\infty,r_1(\delta)) \cup [r_1(\delta),r_2(\delta))\cdots \cup[r_{n'}(\delta),\infty).
\end{equation}

Using \eqref{eq:3.6} and \eqref{eq:3.61}, we can sub-divide each $E_j$ into parts intersecting with the intervals $I_k(\delta)$, by defining 
\begin{equation}\label{eq:3.7}
    E_{jk}(\delta) =E_j \cap I_k(\delta) = E \cap \tilde{I}_j \cap I_k(\delta).
\end{equation}

There are two cases to consider
\begin{equation}\label{eq:3.8}
    \tilde{I}_j \cap I_k(\delta)  = I_{jk}(\delta) = \begin{cases}
    \bigg[\tr_{j-1}\vee r_{k-1}(\delta), \tr_j\wedge r_k(\delta)\bigg),\;\;\;\text{if}\;\; \tilde{I}_j \cap I_k(\delta) \neq \emptyset,\\\\
    \emptyset,\;\;\;\;\;\;\;\;\;\;\;\;\;\;\;\;\;\;\;\;\;\;\;\;\;\;\;\;\;\;\;\;\;\;\;\;\;\;\;\;\;\;\;\text{if}\;\; \tilde{I}_j \cap I_k(\delta) = \emptyset,
    \end{cases}
\end{equation}
where $\tr_{j-1}\vee r_{k-1}(\delta) = \text{max}\big(\tr_{j-1}, r_{k-1}(\delta)\big)$ and $\tr_j\wedge r_k(\delta) = \text{min}\big(\tr_{j}, r_{k}(\delta)\big).$ It is important to note that the sign of $h_\delta(X)$ remains constant on each interval $I_k(\delta)$ and hence also on $I_{jk}$. If $j\in J^+$
\begin{equation}\label{eq:3.9}
E_{jk}(\delta) = \begin{cases}
 I_{jk}(\delta) \;\;\; \text{if}\;\; h_\delta(X)\leq 0  \;\text{on}\;I_k(\delta),\\ \\
\emptyset \;\;\;\;\; \;\;\; \text{otherwise}.
\end{cases}
\end{equation}
If $j\in J^-$
\begin{equation}\label{eq:3.10}
E_{jk}(\delta) = \begin{cases}
I_{jk}(\delta) \;\;\; \text{if}\;\; h_\delta(X)> 0  \;\text{on}\;I_k(\delta),\\ \\
\emptyset \;\; \text{if}\;\;\; \text{otherwise}.
\end{cases}
\end{equation}
Hence, we can refine \eqref{eq:3.5} and compute the threshold risk probability for the rational function $h(X)$ as 
\begin{equation}\label{3.11}
R(\delta) = \sum_{j\in J^+}\sum_{k=1}^{n'+1}P(I_{jk}(\delta)) + \sum_{j\in J^-}\sum_{k=1}^{n'+1}P(I_{jk}(\delta)).
\end{equation}
Let 
\begin{align*}
K^- & = \bigg\{k\;|\;\; h_\delta(X) \leq 0 \;\; \text{on}\;\; I_k(\delta) \bigg\}\\
K^+ & = \bigg\{k\;|\;\; h_\delta(X)>0 \;\; \text{on} \;\;I_k(\delta) \bigg\}.    
\end{align*}
Substituting \eqref{eq:3.9} and \eqref{eq:3.10} into \eqref{3.11} and rearranging, we obtain the risk probability
\begin{equation}\label{eq:3.12}
    \begin{split}
    R(\delta) &=    \sum_{j\in J^+}\sum_{k \in K^-}P(E_{jk}(\delta)) +\sum_{j\in J^-}\sum_{k \in K^+}P(E_{jk}(\delta)).\\
    &= \sum_{j\in J^+}\sum_{k \in K^-}\bigg[F\big(\tr_j \wedge r_k(\delta) \big)-F\big(\tr_{j-1} \vee r_{k-1}(\delta) \big) \bigg]\\
    & + \sum_{j\in J^-}\sum_{k \in K^+}\bigg[F\big(\tr_j \wedge r_k(\delta) \big)-F\big(\tr_{j-1} \vee r_{k-1}(\delta) \big) \bigg],
    \end{split}
\end{equation}
and its derivative, whenever it exists.
\begin{equation}\label{eq:3.13}
    \begin{split}
    R'(\delta) &= \sum_{j\in J^+}\sum_{k \in K^-}\bigg[f\big(\tr_j \wedge r_k(\delta) \big)\frac{d}{d\delta}\big(\tr_j \wedge r_k(\delta) \big)-f\big(\tr_{j-1} \vee r_{k-1}(\delta) \big)\frac{d}{d\delta}\big(\tr_{j-1} \vee r_{k-1}(\delta) \big)\ \bigg]\\
    & + \sum_{j\in J^-}\sum_{k \in K^+}\bigg[f\big(\tr_j \wedge r_k(\delta) \big)\frac{d}{d\delta}\big(\tr_j \wedge r_k(\delta) \big)-f\big(\tr_{j-1} \vee r_{k-1}(\delta) \big)\frac{d}{d\delta}\big(\tr_{j-1} \vee r_{k-1}(\delta) \big) \bigg].
    \end{split}
\end{equation}

Equations \eqref{eq:3.12}-\eqref{eq:3.13} are analogous to \eqref{eq:2.3}-\eqref{eq:2.4} in the one polynomial case. Naturally, they reflect an additional degree of complexity arising from the possibility of overlaps between intervals $\tilde{I}_j$ and $I_k(\delta)$. While this complexity is unavoidable, one case where there are no difficulties is described in the next lemma.

\begin{lemma}
Let $\mathcal{Z}(h_{\delta}(X)) = \{ \delta \; | \text{Dis}(h_{\delta}(X)) = 0 \}$ and consider the intervals $\tilde{I}_j$ in \eqref{eq:3.2} and $I_k(\delta)$ in \eqref{eq:3.61}. Then  $\delta$ is not a candidate for risk critical point if each interval $\tilde{I}_j$ for $j \in J^+$ is contained in some interval  $I_k$ for  $k \in K^-$ or similarly if each interval $\tilde{I}_j$ for $j \in J^-$ is contained in some interval  $I_k(\delta)$ for  $k \in K^+$.
\end{lemma}

\begin{proof}
Under the interval inclusion hypotheses it can be easily verified that for every $j,$ $\tr_j \wedge r_k(\delta) = \tr_j$ and $\tr_{j-1} \vee r_{k-1}(\delta) = \tr_{j-1}$.
Hence equation \eqref{eq:3.13} is well-defined and reduces to  
\begin{equation*}
    \begin{split}
    R'(\delta) &= \sum_{j\in J^+}\sum_{k \in K^-}\bigg[f\big(\tr_j  \big)\frac{d}{d\delta}\big(\tr_j\big)-f\big(\tr_{j-1} \big)\frac{d}{d\delta}\big(\tr_{j-1} \big)\ \bigg]\\
    & + \sum_{j\in J^-}\sum_{k \in K^+}\bigg[f\big(\tr_j  \big)\frac{d}{d\delta}\big(\tr_j \big)-f\big(\tr_{j-1}  \big)\frac{d}{d\delta}\big(\tr_{j-1} \big) \bigg]\\
    & = 0,\\
    \end{split}
\end{equation*}
since none of the terms depend on $\delta$. Thus the derivative of the risk is $0$ and hence $\delta$ cannot be a risk critical point. 
\end{proof}

\section{Hidden equation of polynomially perturbed case}
In this section, we are going to discuss the hidden equations of the perturbed polynomial in the form $h_\delta(X) = p(X) - \delta q(X)$. As before,  assume $p(X) =p_0+p_1X+\ldots +p_nX^n$ and $q(X) = q_0 + q_1X+\ldots q_mX^m$ are polynomials. This is a more complicated case compared to the constant perturbation discussed in Section 3.
\begin{lemma}
Let $h_\delta(X) = p(X)-\delta q(X)$ where $\text{deg}(p(X))=n, \text{deg}(q(X))=m$. The maximum order of the hidden polynomial $\text{Dis}(h_\delta(X))(\delta)$ has the following cases:
\begin{enumerate}
    \item $\text{deg}(\text{Dis}(h_\delta(X))(\delta)) \leq 2m-2$ if $m>n$ and $q_m \delta \neq 0$,
    \item $\text{deg}(\text{Dis}(h_\delta(X))(\delta)) \leq 2n-2$ if $m<n$ and $p_n \neq 0$,
    \item $\text{deg}(\text{Dis}(h_\delta(X))(\delta)) \leq 2n-2$ if $m =n$ and $p_n - q_n \delta \neq 0.$ 
\end{enumerate}
\end{lemma}
\begin{proof}
Suppose $m> n$. Using Lemma 6 in the Appendix, we can compute the discriminant in $\delta$ as follows, 
\begin{equation}\label{eq:4.1}
\begin{split}
      &\text{Dis}(h_\delta(X)) =\frac{(-1)^{n(n-1)/2}}{(-q_m\delta)}\text{Res}(h_{\delta}(X),h'_{\delta}(X))\\ &=(-1)^{n(n-1)/2}(-q_m\delta)^{-1} \times\\
      &\text{Det}\begin{bmatrix}-q_m\delta&-q_{m-1}\delta&\cdots&\cdots&p_0-q_0\delta&0&\cdots&0\\
    0&-q_{m}\delta&\cdots&\cdots&\cdots&p_0-q_0\delta&\cdots&0\\
    \vdots&\vdots&\ddots&\ddots&\vdots&\vdots&\ddots&p_0-q_0\delta\\
    -mq_m\delta&-(m-1)q_{n-1}\delta&\cdots&\cdots&p_1-q_1\delta&0&\cdots&0\\
    0&-mq_{m}\delta&(m-1)q_{m-1}\delta&\cdots&\cdots&p_1-q_1\delta&\cdots&0\\
        \vdots&\vdots&\ddots&\ddots&\vdots&\vdots&\ddots&p_1-q_1\delta\\
    \end{bmatrix}.
\end{split}
\end{equation}

Note that every row depends linearly on $\delta$. We use the multi-linearity of determinant of the $(2m-1) \times (2m-1)$ Sylvester's matrix. We observe that it is a polynomial in $\delta$ of degree $2m-1$. Multiplying it by $(-1)^{n(n-1)/2}(-q_m\delta)^{-1} \neq 0$, it becomes a polynomial in $\delta$ with maximum degree of $2m-2$. The proofs of the other cases are very similar.
\end{proof}

Let $\mathcal{Z}(h_\delta(X)) = \mathcal{Z}_m(h_\delta(X))\cup \mathcal{Z}'(h_\delta(X))$. The set $\mathcal{Z}_m(h_\delta(X))$ is the set of zeroes of the leading coefficient of $h_\delta(X)$,  and $\mathcal{Z}'(h_\delta(X)) = \mathcal{Z}(h_\delta(X))\setminus\mathcal{Z}_m(h_\delta(X))$ be the set of $\delta$ which are zeroes of the discriminant but not of the leading coefficient. Using \eqref{eq:a.7} in Definition 8 of the Appendix, we note that $\delta$ can be a zero of the discriminant of $h_\delta(X)$ either by being a zero of the resultant or a zero of the leading coefficient. The latter occurs when $\delta \in \mathcal{Z}_m(h_\delta(X))$.\\

With the help of Lemma 4, we can find the candidates for the risk critical points by solving the hidden equations  in $\delta$. However, as we have seen in the previous sections, not all roots of the hidden equations are guaranteed to be risk critical points. 
% Using Theorem 1, if $\delta \in \mathcal{Z}'(h_\delta(X))$, where the induced roots $r(\delta)$  have branching order greater than 2 and the density function $f(r(\delta))>0$, then $\delta$ is a risk critical point.
For instance, if $\delta \in \mathcal{Z}_m(h_\delta(X))$, we have the following corollary. 

\begin{corollary}
If $\delta_0 \in \mathcal{Z}_m(h_{\delta}(X))$, then $\delta_0$ is a candidate for risk critical point irrespective of the branching order of the root $r_j(\delta_0)$.
\end{corollary}
\begin{proof}
Using Theorem 3(c) in the Appendix, if $\delta_0 \in \mathcal{Z}_m(h_{\delta}(X))$, then $\delta_0$ is zero of the leading coefficient of $h_{\delta_0}(X)$ with multiplicity 1 because the leading coefficient of $h_{\delta_0}(X)$ is linear in $\delta_0$. The root $r_j(\delta)$ has a Laurent-Puiseux series representation 
\begin{equation}\label{eq:4.2}
r_j(\delta) = \sum_{k=-1}^{\infty} c_k(\delta-\delta_0)^{k/m'},
\end{equation}
where $m'>0$ is the order of the branching point. If we take the derivative of the first term of $r_j(\delta)$, we have 
$\frac{-c_{-1}}{m'}(\delta-\delta_0)^{\frac{-1}{m'}-1}$, this will diverge for any $m'>0$ as $\delta \to \delta_0$.
\end{proof}

$\textbf{Remark 3:}$ In this section the perturbed polynomial was of the form $h_\delta(X) = p_0(\delta)+p_1(\delta)X+\ldots+p_n(\delta)X^n.$ In this case there are two hidden equations associated with characterization of risk critical points. The first, as before, is \eqref{eq:dis} and the second where the leading coefficient is 0, namely \begin{equation}\label{eq:dis2}\tag{$**$}
    p_n(\delta) = 0.
\end{equation}

% $\textbf{Remark 4:}$ We note that the assumption of absolute continuity of the random variable $X$ can be easily relaxed to include points of discontinuity in its cumulative distribution function. For instance, if $\tilde{x}$ is such a point, then $\delta = \tilde{x}$ is automatically a risk critical point.  

% we merely need to identify whether it lies in one of the intervals $I_j(\delta)$ or whether it coincides with one of the repeated roots. From there, the analysis can proceed as before.  

\section{Illustration via Simulations}
In this section, we present examples demonstrating some of the key results derived earlier. In particular, we show the importance of the connection between the distribution of the underlying random variable and the location of the roots of the perturbed polynomials. We are going to show three numerical examples of risk critical points. The probabilities of the interval events $I_j(\delta)$ were calculated using Monte Carlo method and roots of the perturbed polynomials were derived manually.\\ 

\textbf{Example 2:} Let $p_\delta(X) = X^2(X-2)-\delta$. Here the hidden equation is $\text{Dis}(p_\delta(X)) = -\delta(27\delta+32) =0$, and we have two candidates for the risk critical points, $\delta=0$ and $\delta = -\frac{32}{27}.$ When $\delta=0$, $X=0$ is a root with even multiplicity and when $\delta = -\frac{32}{27}$, $X = 4/3$ is a root with even multiplicity. Hence we simulated the threshold risk using random variables $X\sim \mathbb{N}(0,1^2)$ and $X \sim \mathbb{N}(4/3,1^2)$ respectively. 
\newpage

\begin{figure}[ht!]
    \centering
    \subfloat[$X\sim \mathbb{N}(0,1^2)$]{{\includegraphics[width=5cm]{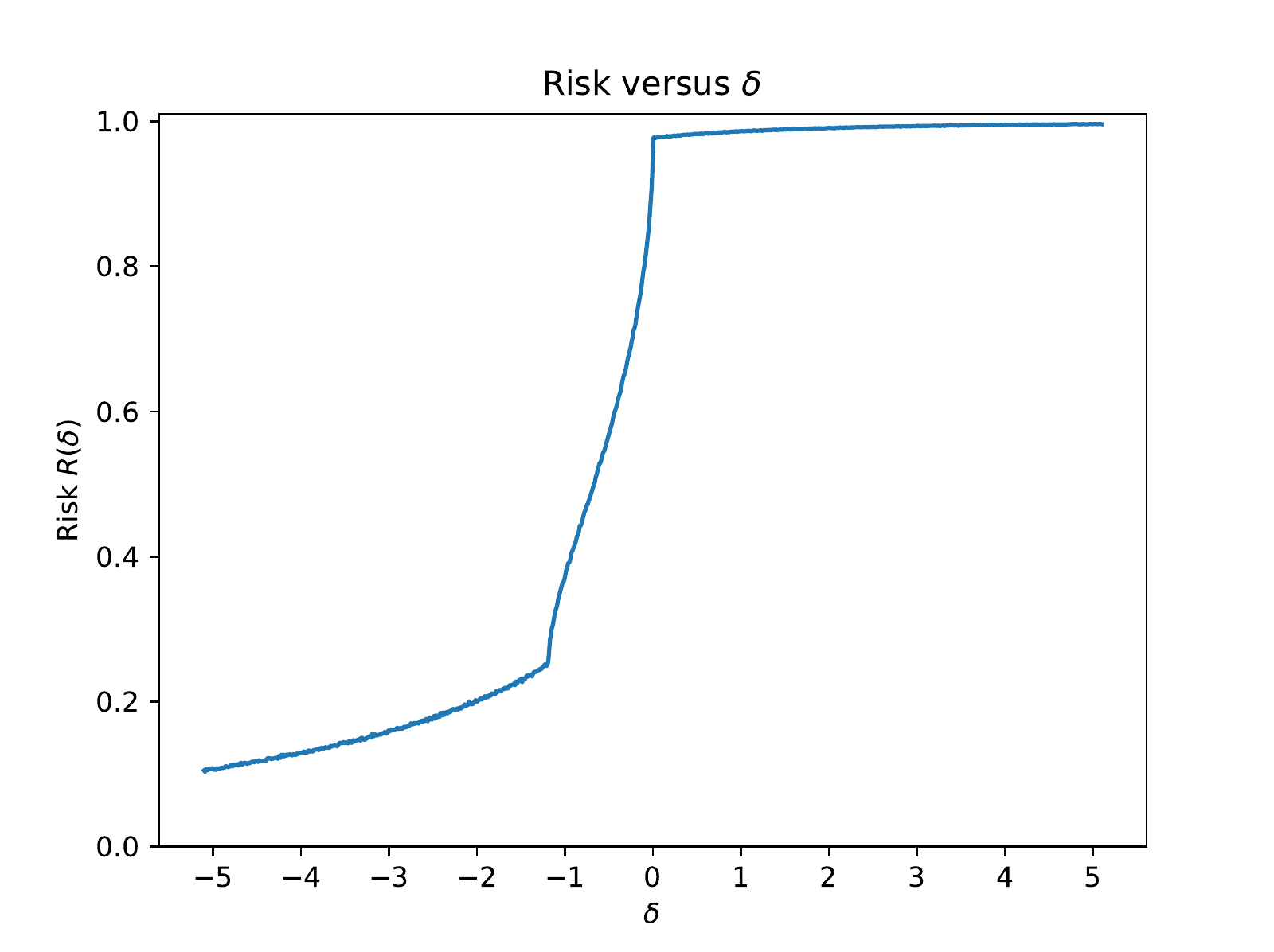} }}
    \subfloat[$X \sim \mathbb{N}(4/3,1^2)$]{{\includegraphics[width=5cm]{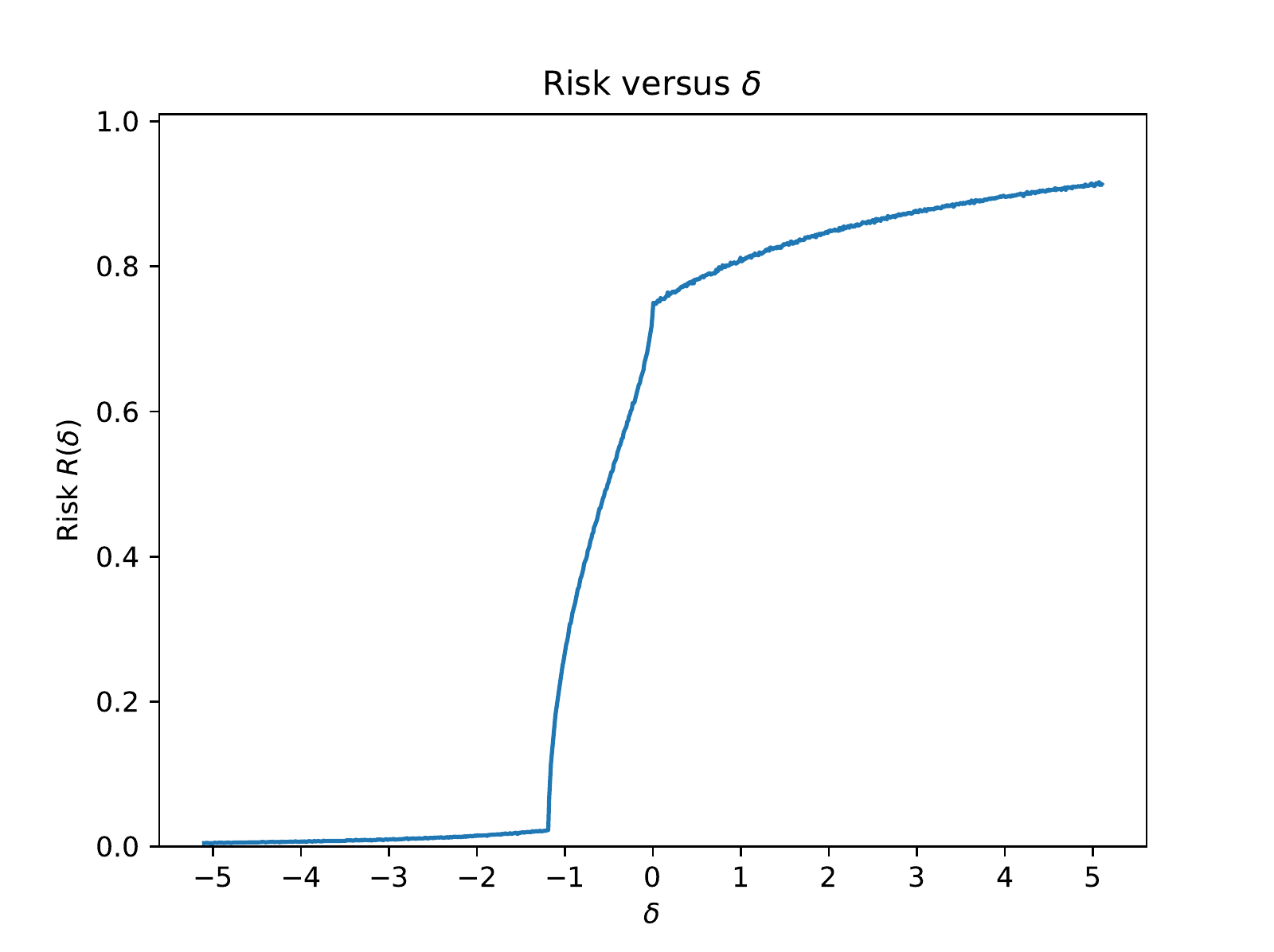} }}
    \caption{Risk versus $\delta$ for $p_\delta(X) = X^2(X-2)-\delta$}%
    \label{fig:case1}%
\end{figure}
Figure 1 demonstrates high sensitivity of  $R(\delta)$ in the neighbourhoods of $\delta = 0 $ and $\delta = -\frac{32}{27}.$\\

\textbf{Example 3:} Let $p_\delta(X) = X^2-\delta X$. Here the hidden equation is $\text{Dis}(p_\delta(X)) = \delta^2 =0$, we have one candidate for a risk critical point, $\delta=0.$  When $\delta=0$,  $X=0$ is a root of $p_\delta(X)$. Hence we simulated the risk using random variable $X\sim \mathbb{N}(0,1^2)$. \\

\begin{figure}[ht!]%
    \centering
    \includegraphics[width=7cm]{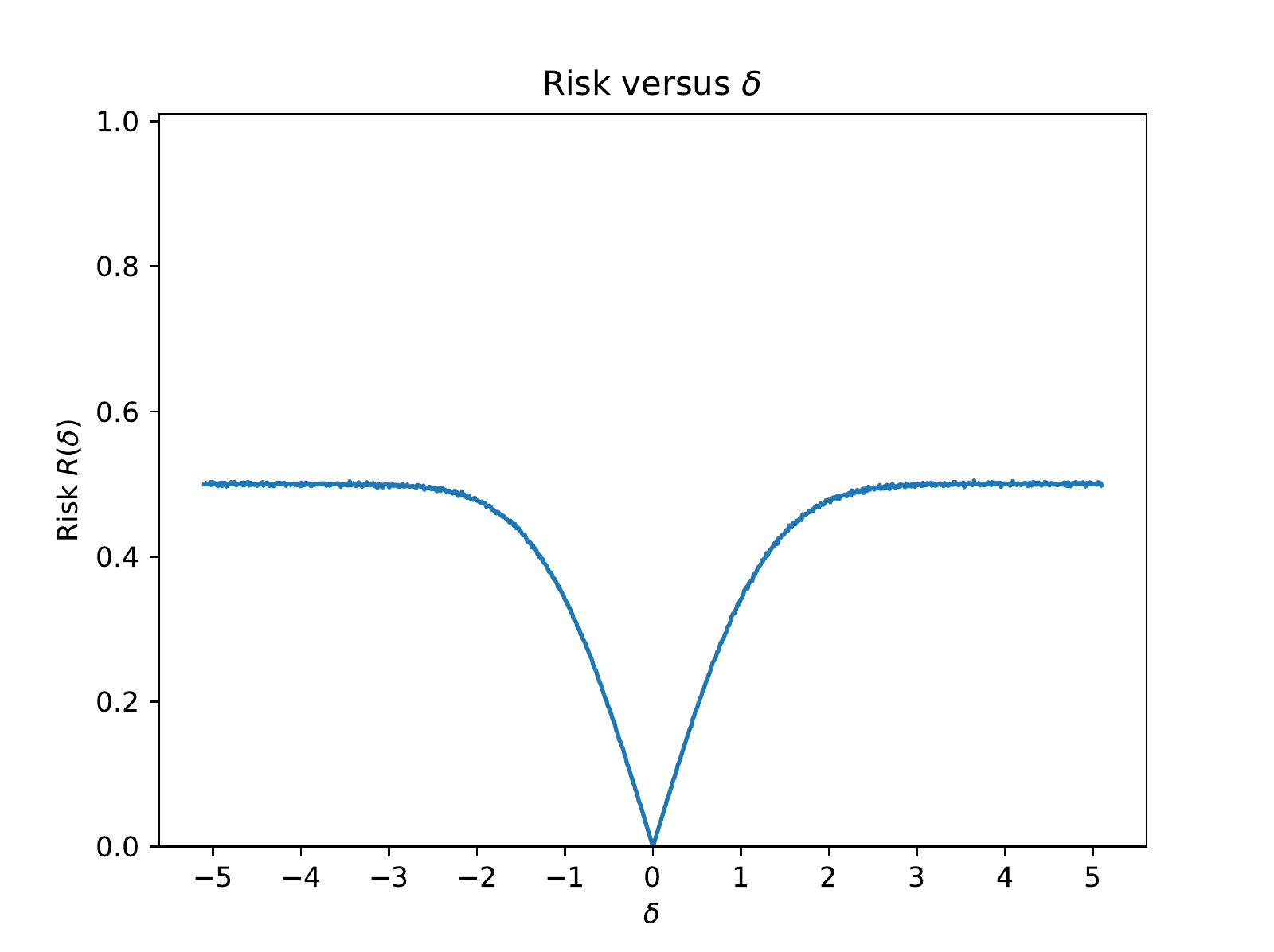}
        \caption{Risk versus $\delta$ for $p_\delta(X) = X^2-\delta X$}%
    \label{fig:case2}
\end{figure}

Figure 2 once again shows the sensitivity of threshold risk in the neighbourhood of $\delta =0$. However, note that when $\delta$ increases from $0$, by the symmetry of normal distribution, the threshold risk increases rapidly from $0$ towards $0.5$. Similarly, for $\delta$ decreasing from $0$.

\textbf{Example 4:} Let $p_\delta(X) = X-\delta X^2$. Here the $\text{Dis}(p_\delta(X)) = 1$, hence the hidden equation \eqref{eq:dis} has no roots. In accordance with Corollary 2, the remaining hidden equation \eqref{eq:dis2} corresponds to the leading coefficient becoming 0. In this case, we have only one candidate for the risk critical point, $\delta =0$. However, note that as $\delta$ approaches $0$ from above, the non-zero root of $p_\delta(X)=0$ approaches $\infty$. Hence, the sensitivity of the threshold risk only manifests itself for distributions with sufficiently heavy tails. Hence we simulate the risk using random variable $X\sim \text{Cauchy}(x_0=0,\gamma= 1)$.
\begin{figure}[ht]
    \centering
    \includegraphics[width=6cm]{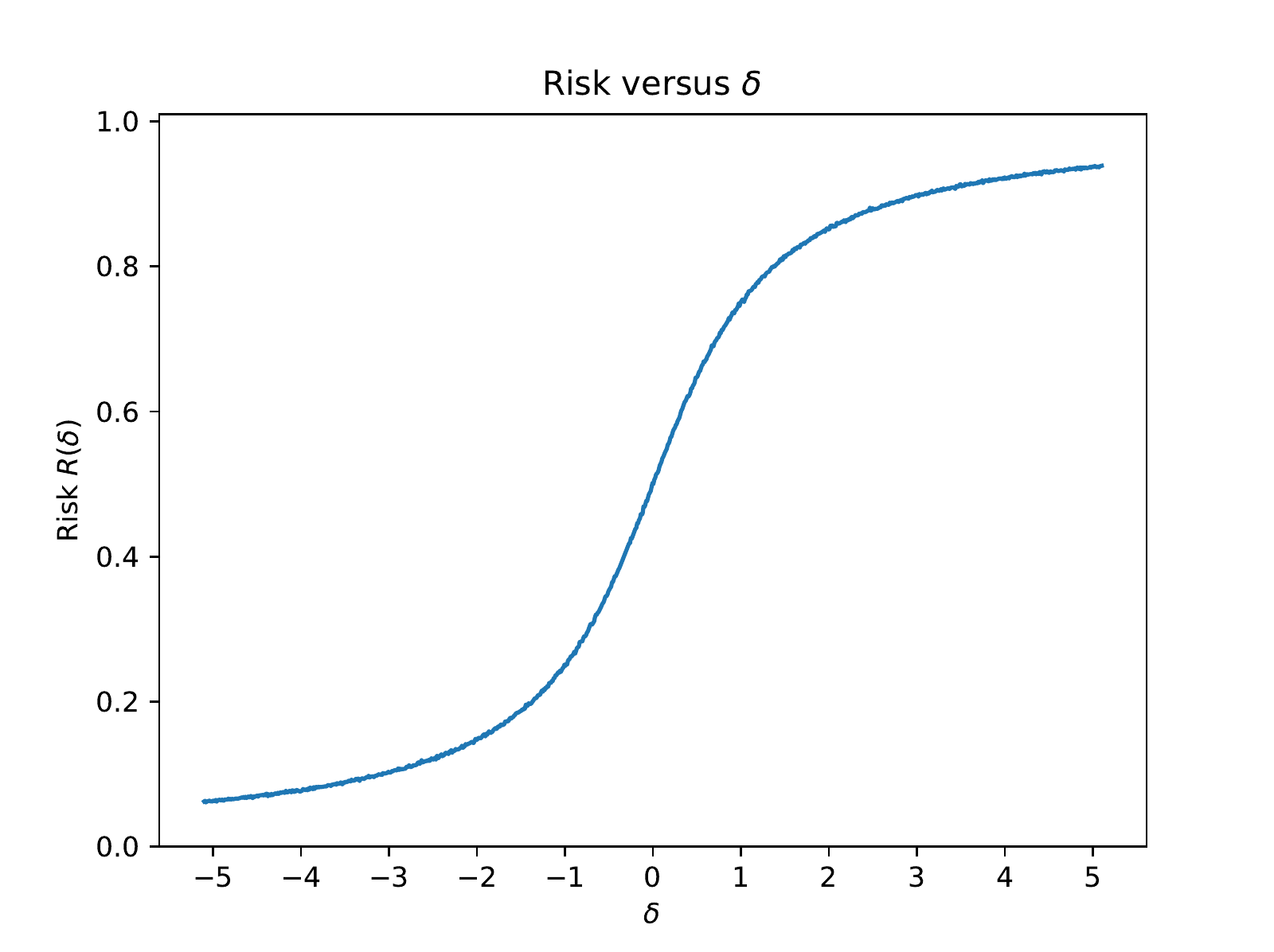}
        \caption{Risk versus $\delta$ for $p_\delta(X) = X-\delta X^2$}%
    \label{fig:case3}
\end{figure}\\
Figure 3 exhibits the sensitivity of threshold risk in the neighbourhood of the risk critical point $\delta=0$.

\newpage
\appendix
\section{}
For the sake of completeness, we recall a number of important relationships involving polynomials, resultants and the discriminant. While proofs of some of these results can be found in many sources we cite mainly the widely used reference \cite{gelfand}. We also cite \cite{avrachenkov_analytic_2013} because the latter contains the proof of Theorem 3 that is not easily found elsewhere.

\begin{definition}
For real polynomials $f(x) = a_0 + a_1x +\ldots + a_nx^n$ and $g(x
) = b_0+b_1x+\ldots + b_m x^m$, with $\text{deg}(f) =n, \text{deg}(g) = m$, their resultant $\text{Res}(f,g)$ is the determinant of the $(m+n) \times (m+n) $ Sylvester matrix, given by \begin{equation}\label{eq:a.1}
    \text{Res}(f,g) = \text{Det}\begin{bmatrix}a_n&a_{n-1}&\cdots&\cdots&a_0&0&\cdots&0\\
    0&a_{n}&\cdots&\cdots&a_1&a_0&\cdots&0\\
    \vdots&\vdots&\ddots&\ddots&\vdots&\vdots&\ddots&a_0\\
    b_m&b_{m-1}&\cdots&\cdots&b_0&0&\cdots&0\\
    0&b_{m}&b_{m-1}&\cdots&\cdots&b_0&\cdots&0\\
        \vdots&\vdots&\ddots&\ddots&\vdots&\vdots&\ddots&b_0\\
    \end{bmatrix}.
\end{equation}
\end{definition}

\begin{theorem}
For real polynomials $f(x) = a_0 + a_1x +\ldots + a_nx^n$ and $g(x
) = b_0+b_1x+\ldots + b_m x^m$, suppose that $f $ has roots $\alpha_1,\ldots,\alpha_n$ and $g$ has roots $\beta_1,\ldots,\beta_m$ (not necessarily distinct). Then the resultant can be computed as 
\begin{equation}\label{eq:a.2}
    \text{Res}(f,g) = a_{n}^m b_{m}^n \prod_{i=1}^{n} \prod_{j=1}^{m} (\alpha_i - \beta_j). 
\end{equation}
\end{theorem}
\begin{proof}
See reference \cite{gelfand} page 408.
\end{proof}
\begin{lemma}
For real polynomials $f(x) = a_0 + a_1x +\ldots + a_nx^n$ and $g(x
) = b_0+b_1x+\ldots + b_m x^m$, where $m\leq n$, suppose that $f $ has roots $\alpha_1,\ldots,\alpha_n$ and $g$ has roots $\beta_1,\ldots,\beta_m$ (not necessarily distinct). Then the resultant can be computed as
\begin{equation}\label{eq:a.3}
    \text{Res}(f,g) = a_n^m \prod_{i=1}^{n}g(\alpha_i),
\end{equation}
where $g(x) = b_m\prod_{j=1}^{m}(x-\beta_j)$, and $g(\alpha_i) = b_m\prod_{j=1}^{m}(\alpha_i-\beta_j)$.
\end{lemma}
\begin{proof}
Follows immediately from Theorem 2.
\end{proof}

\begin{definition}
Let $f(x)=a_0 + a_1 x + a_2 x^2 +\ldots + a_n x^n$ be a real polynomial, the discriminant of $f$ is 
\begin{equation}\label{eq:a.4}
    \text{Dis}(f) = a_n^{2n-2} \prod_{1\leq i \leq j \leq n } (\alpha_i - \alpha_j)^2.
\end{equation}
where $\alpha_1,\ldots,\alpha_n$ are the roots of $f$ (not necessarily distinct).  
\end{definition}

\begin{lemma}
Let $f = a_0 + a_1x + a_2x^2+\ldots+a_nx^n$, the discriminant of $f$ is given by 
\begin{equation}\label{eq:a.5}
    \text{Dis}(f) = (-1)^{n(n-1)/2}a_{n}^{-1}\text{Res}(f,f').
\end{equation}
\end{lemma}
\begin{proof}
Follows from equation (1.23) on Page 404 of \cite{gelfand}.
% Let $f(x)$ have roots $\alpha_1,\ldots,\alpha_n$, then $f(x) = a_n \prod_{i=1}^{n} (x-\alpha_i)$, thus $f'(x) =a_n \sum_{i=1}^{n}\prod_{j\neq i }^{n}(x-\alpha_j) $, in particular, we have $f'(\alpha_i) = a_n \prod_{j\neq i} (\alpha_i-\alpha_j)$, since $\text{deg}(f') =n-1 < \text{deg}(f)$, we can use \eqref{eq:a.3},  we have 
% \begin{equation*}
% \begin{split}
% R(f,f') &= a_{n}^{n-1}\prod_{i}^{n}f'(\alpha_i)\\
% & = a_{n}^{n-1}\prod_{i}^{n}a_n\prod_{j\neq i}(\alpha_i - \alpha_j)\\
% & = a_{n}^{2n-1}\prod_{1\leq j\leq i\leq n}(\alpha_i-\alpha_j)(\alpha_j - \alpha_i)\\
% & = a_{n}^{2n-1}(-1)^{n(n-1)/2}\prod_{1\leq j\leq i\leq n}(\alpha_i-\alpha_j)^2\\
% & = a_n(-1)^{n(n-1)/2}\Delta(f).
% \end{split}
% \end{equation*}
% Hence, we proved that $\Delta(f) =a_n^{-1}(-1)^{n(n-1)/2}R(f,f'). $
\end{proof}
% \begin{lemma}
% If $f(x) = a_0 + a_1x+a_2x^2+\ldots + a_nx^n$ has a multiple roots, then $\Delta(f) = 0$.
% \end{lemma}
% \begin{proof}
% Directly follows from \eqref{eq:a.4}.
% \end{proof}

\begin{definition} Let
\begin{equation}\label{eq:a.6}
    Q(x,z) = q_n(z)x^n + q_{n-1}(z) x^{n-1} + \ldots+q_0(z)
\end{equation}
$Q(x,z)$ is a bivariate polynomial with the perturbation variable $z$.  Using \eqref{eq:a.4}, the discriminant of $Q(x,z)$ has the following form 
\begin{equation}\label{eq:a.7}
\text{Dis}(Q,z) = q_{n}(z) \prod_{i<j}(\alpha_i(z)-\alpha_j(z))^2,   
\end{equation}
where $\alpha_1(z), \ldots, \alpha_n(z)$ are the roots of $Q(x,z)$. 
\end{definition}
Let $\mathcal{Z}(Q) = \mathcal{Z}_n(Q) \cup \mathcal{Z}'(Q)$ be the zero set of $\text{Dis}(Q,z_0)$. More specifically,  $\mathcal{Z}_n(Q) = \{z | q_n(z) = 0\}$ and $\mathcal{Z}'(Q) = \{z | \text{Dis}(Q,z)=0, q_n(z) \neq 0 \}$. The following theorem provides the algebraic analytic form of the root function $x=x(z)$ in various situations with respect to the nature of the point $z$. We note that, in some cases, the latter is an analytic multi-valued function  $f(z)$ defined in a punctured neighborhood of $z$ satisfying $Q(f(z), z)=0$ for all $z$ in the complement of $\mathcal{Z}(Q)$. In those cases, the type of series expansion that results depends on the limiting properties of $f$ when z approaches $z$, as stated more precisely in the theorem.

\begin{theorem}(Classification of root expansions \cite{avrachenkov_analytic_2013})
\begin{enumerate}
    \item[(a)] If $z_0 \notin \mathcal{Z}(Q)$ and is  not a zero of $q_n(z)$, then in a neighborhood of $z_0$ every one of the $n$ branches of the solution $x(z)$ is holomorphic, and so it has the analytic representation 
    \begin{equation}\label{eq:a.8}
    x(z) = \sum_{k=0}^{\infty} c_k (z-z_0)^k.
    \end{equation}
    \item[(b)] If $z_0 \in \Z'(Q)$, then $z_0$ is a branching point of some order $n' \leq n$ for every branch $f(z)$ of the solution $x(z)$ and also $\lim_{z \rightarrow z_0}f(z) =0$. In this case the solution $x(z)$ has a Puiseux series representation 
    \begin{equation}\label{eq:a.9}
    x(z) = \sum_{k=0}^{\infty} c_k (z-z_0)^{k/n'}.
    \end{equation}
    \item[(c)] If $z_0 \in \Z_n(Q)$ and is a zero of multiplicity $n_0>0$ of $q_n(z)$, then for any branch $f(z)$ of $x(z)$ the point $z_0$ is a branching point of some order $n'\leq n$ and $\lim_{z \rightarrow z_0}(z-z_0)n^{n+\delta}f(z)=0$ for all $\delta>0$. In this situation the solution $x(z)$ has a Laurent-Puiseux series representation 
    \begin{equation}\label{eq:a.10}
        x(z) = \sum_{k=-k_0}^{\infty} c_k (z-z_0)^{k/n'}.
    \end{equation}
    \item[(d)] If $z_0 \notin \Z(Q)$ and is the zero of multiplicity $m_0 >0$ of $q_m(z)$, then $z_0$ is a pole or order $m_0$ for every branch $f(z)$ of the solution $x(z)$, and in this situation the solution $x(z)$ has a Laurent series representation
    \begin{equation}\label{eq:a.11}
        x(z) = \sum_{k=-m_0}^{\infty}c_k(z-z_0)^k.
    \end{equation}
\end{enumerate}
\end{theorem}
\begin{proof}
See Theorem 4.8 on Page 93 \cite{avrachenkov_analytic_2013}. 
\end{proof}





\bibliography{ref} 
\bibliographystyle{chicago}

\end{document}